\documentclass{llncs}

\usepackage[english]{babel}
\usepackage[utf8]{inputenc}
\usepackage[T1]{fontenc}
\usepackage{amsfonts}
\usepackage{stmaryrd}
\usepackage{amsmath,amssymb}
\usepackage{array}
\usepackage{graphicx}
\usepackage{nicefrac}
\usepackage{enumerate}
\usepackage{caption}
\usepackage{color}
\usepackage{tikz}
\usepackage{url}

\newcommand{\lsim}{\mathrel{\hbox{\rlap{\lower.55ex \hbox{$\sim$}} \kern-.3em \raise.7ex \hbox{$<$}}}}
\newcommand{\gsim}{\mathrel{\hbox{\rlap{\lower.55ex \hbox{$\sim$}} \kern-.3em \raise.7ex \hbox{$>$}}}}

\newcommand{\mc}[1]{\ensuremath{\mathcal{#1}}}
\newcommand{\bs}[1]{\ensuremath{\boldsymbol{#1}}}

\def\ao{$\mathsf{and/or}$ }
\def\a{$\mathsf{and}$ }

\def\true{\mathsf{true}}
\def\false{\mathsf{false}}
\def\rat{\mathsf{rat}}

\newcommand{\A}{\mathcal{A}}

\newcommand{\F}{\mathcal{F}}

\newcommand{\I}{\mathcal{I}}
\newcommand{\M}{\mathcal{M}}

\newcommand{\T}{\mathcal{T}}

\newcommand{\e}{\mathtt{e}}

\newcommand{\IN}{\mathbb{N}}
\newcommand{\IP}{\mathbb{P}}

\begin{document}

\frontmatter         

\pagestyle{headings}  
\mainmatter      

\title{Catalan Satisfiability Problem\thanks{Partially supported by  the A.N.R. project {\em  BOOLE}, 09BLAN0011.}}

\titlerunning{Catalan Satisfiability Problem} 

\author{Antoine Genitrini\inst{1} \and Cécile Mailler\inst{2}}

\authorrunning{Genitrini et al.} 

\institute{Laboratoire d'Informatique de Paris 6;
\email{Antoine.Genitrini@lip6.fr}.
\and
Laboratoire de Mathématiques de Versailles;
\email{Cecile.Mailler@uvsq.fr}.
}

\maketitle             

\begin{abstract}
An \ao tree is usually a binary plane tree, with internal nodes labelled by logical connectives,
and with leaves labelled by literals chosen in a fixed set of $k$ variables and their negations.
In the present paper, we introduce the first model of such Catalan trees, whose number of
variables $k_n$ is a function of $n$, the size of the expressions.
We describe the whole range of the probability distributions
depending on the functions $k_n$, as soon as it tends jointly with $n$ to infinity.
As a by-product we obtain a study of the satisfiability problem in the context of Catalan trees.

Our study is mainly based on analytic combinatorics and extends the Kozik's
\emph{pattern theory}, first developed for the fixed-$k$ Catalan tree model.

\end{abstract}

\section{Introduction}

Since years many scientists of different areas,
e.g. computer scientists, mathematicians or statistical physicists,
are studying satisfiability problems (like $k$-SAT problems) and some questions that arise around them:
for example,
phase transitions between satisfiable and unsatisfiable expressions or constraints satisfaction problems.
The classical 3-SAT problem takes into consideration expressions of a specific form: conjunction of clauses
that are themselves disjunctions of three literals. The literals are chosen among a finite set whose size 
is linked to the size of the expression. Then one question consists of deciding if a large random expression
is satisfiable or not.
Actually we know among other things, see~\cite{AM06} for example,
that the satisfiability problem is related to the ratio between the size
of the expression and the number of allowed literals. 
There is a phase transition such that, 
when the ratio is smaller than a critical value, 
the random expression is satisfiable with probability tending to~$1$ 
when the size of the expression tends to infinity,
while when the ratio is larger than the critical value, 
the probability tends to~$0$.

An interesting paper~\cite{DR11} about constraints satisfaction problems
deals with random $2-XORSAT$ expressions. Using generating functions, in the context of analytic combinatorics' tools,
the authors describe precisely the phase transition between satisfiable and unsatisfiable expressions.

Still dealing with Boolean expressions, but in a completely distinct direction,
researchers have studied the complete probability distribution on Boolean functions
induced by random Boolean expressions.
The first approach, by Lefmann and Savick\'y~\cite{LS97}, consists in fixing a finite set of variables,
 allowing the two logical connectives \a and $\mathsf{or}$ and choosing uniformly at random a Boolean
  expression of \emph{size} $n$ in this logical system.
Their model is usually called the \emph{Catalan model}.
 Lefmann and Savick\'y first proved the existence of a limiting probability distribution
  on Boolean functions when the size of the random Boolean expressions tends to infinity.
Since the seminal paper by Chauvin et al.~\cite{CFGG04},
 almost all quantitative studies of such Boolean distributions are deeply related to analytic combinatorics:
  a survey by Gardy~\cite{gardy06} provides a wide range of models with various numerical results.
Later, Kozik~\cite{kozik08} proved a strong relation between the limiting probability of a given function and its \emph{complexity}
(i.e. the minimal \emph{size} of an expression representing the function).
His approach lies in two separate steps:
(i) first let the size of the Boolean expressions taken into consideration tend to infinity, and then 
(ii) let the number of variables used to label the expressions tend to infinity.
His powerful machinery, the \emph{pattern theory},
easily classifies and counts large expressions according to structural constraints.
The main objection to this model is about the two consecutive limits that cannot be interchanged: 
the number of allowed variables cannot depend on the size of the expressions.
Genitrini and Kozik have proposed another model~\cite{GKZ07,GK12} that allows to understand
the bias by constructing random Boolean expressions built on an infinite set of variables. 
However, according to our knowledge, no possibility to link the number of variables to the size
has yet been presented, and understanding satisfiability problems in this context is not yet possible.

Our paper extends the Catalan model in order to fit in the satisfiability context.
By using an equivalence relationship on Boolean expressions, we manage to let both
the number of variables and the size of expressions tend jointly to infinity. The number of variables is
a function of the size of the expressions and thus we deal with satisfiability in the context of Catalan expressions.
Furthermore by extending the techniques of Kozik, we describe in details the probability distribution on functions
and exhibit some threshold for the latter distribution:
as soon as the number of variables is \emph{large enough} compared to the size of the expressions,
the general behaviour of the induced probability on Boolean functions does not change by adding more variables.

The paper is organized as follows.
Section~\ref{sec:model} introduces our unified model based
on an equivalence relationship of Boolean expressions. 
Then, Section~\ref{sec:results} states our three main results:
(1) the satisfiability question for random Catalan expressions;
(2) the link between the probability of a class
 of functions and the complexity of the functions taken into account;
(3) the behaviour of the probability related to the dynamic
 between the number of variables and the size of the expressions.
Section~\ref{sec:technical} is devoted to the technical core of the paper.
Finally Section~\ref{sec:proba} applies our approach to \ao trees and proves the main results.

Almost all proofs are given in the appendices.

\section{Probability distributions on equivalence classes of Boolean functions}\label{sec:model}

\subsection{Contextual definitions}

A Boolean function is a function from $\{0,1\}^{\IN}$ into $\{0,1\}$.
The set of Boolean functions is denoted by $\F$. In the following, $\{x_1, x_2, \dots\}$
will be an element of $\{0,1\}^{\IN}$.
 A variable $x_i$ can be negated ($\bar{x}_i = 1-x_i$), and we call {\bf literal} a variable or its negation.
The two connectives taken into account, \a and $\mathsf{or}$, are respectively denoted by $\land$ and $\lor$.

An \ao Boolean expression is seen as an \emph{\ao tree} i.e. a binary plane tree with leaves labelled
by a literal and with internal nodes labelled by connectives. Each \ao tree computes (or represents)
 a Boolean function. Obviously an infinite number of \ao trees compute
the same Boolean function. The {\bf size} of an \ao tree is its number of leaves: remark that,
for all $n\geq 1$, there is an infinite number of \ao trees of size $n$.

The {\bf complexity} of a Boolean function $f$, denoted by $L(f)$, is defined as the size of its {\bf minimal trees}
, i.e. the smallest trees computing $f$. Although a Boolean function is defined on an infinite set of variables,
 it may \emph{actually} depend only on a finite subset of \emph{essential variables}: given a Boolean function $f$,
  we say that the variable $x$ is {\bf essential} for $f$, if and only if $f_{|x\leftarrow0} \not\equiv f_{|x\leftarrow1}$
   (where $f_{|x\leftarrow \alpha}$ is the restriction of $f$ to the subspace of $\{0,1\}^{\mathbb{N}}$ where $x=\alpha$). 
   We denote by $E(f)$ the number of essential variables of $f$. 
Remark that the complexity and the number of essential variables of a Boolean function are only related by the following
inequality: $E(f)\leq L(f)$.

\subsection{Equivalence relationships}

Analytic combinatorics' tools (cf.~\cite{FS09}) are based on the notion of combinatorial classes.
A \emph{combinatorial class} is a denumerable (or finite) set of objects on which a size notion is defined
such that each object has a non-negative size and the set of objects of any given size is finite.
Thus our class of \ao trees is not a combinatorial class since there is an infinite number of trees of a given size.
To use analytic combinatorics, we define an equivalence relationship on Boolean trees.

In the rest of the paper, we define a {\bf tree-structure} to be an \ao tree in which leaves labels have been removed (but internal nodes remain labelled).
\begin{definition}
Let $A$ and $B$ be two \ao trees.
Trees $A$ and $B$ are {\bf equivalent} if 
(1) their tree-structures are identical, 
if (2) two leaves are labelled by the same variable in $A$ if and only of they are labelled by a same variable in $B$, and 
if (3) two leaves are labelled by the same literal in $A$ if and only of they are labelled by a same literal in $B$.
\end{definition}
This equivalence relationship on Boolean trees \emph{induces straightforwardly
an equivalence relationship on Boolean functions}.
For example, both functions $(x_i)_{i\geq 1}  \mapsto  \bar{x}_{2013}$ and $(x_i)_{i\geq 1} \mapsto x_1$ are equivalent.
An important remark is that all functions of an equivalence class have the same complexity and the same number of essential variables.
In the following, we will denote by $\langle f\rangle$ the equivalence class of the Boolean function $f$.

\subsection{Probability distribution}

Let $\bs{(k_n)_{n\geq 1}}$ {\bf be an increasing sequence that tends to infinity} when $n$ tends to infinity. In the following, we only consider trees such that: for all $n\geq 1$, {\bf the set of variables that appear as leaf-labels} (negated or not) {\bf of a tree of size $\bs n$ has cardinality at most $\bs{k_n}$}.
Remark that if $k_n \geq n$ for all $n\geq 1$, this hypothesis is not a restriction. Therefore, we will assume that $k_n\leq n$, for all $n\geq 1$.

\begin{definition}
We denote by $T_{n}$ the number of equivalence classes of trees of size $n$
 in which at most $k_n$ different variables appear as leaf-labels.
We define the ordinary generating function $T(z)$ as $T(z) = \sum_n T_{n} z^n$.
\end{definition}

\begin{proposition}\label{prop:enum}
The number of classes of trees of size $n$ satisfies:
\[T_{n} = C_n \cdot \sum_{p=1}^{k_n} {n \brace p} 2^{2n-1-p},\]
where $C_n$ is the number of non labelled binary trees\footnote{In Proposition~\ref{prop:enum},
$C_n$ is the $(n-1)$th Catalan number (see e.g. \cite[p. 6--7]{FS09}).}
of size $n$ and ${n \brace p}$ is the Stirling number of the
second kind\footnote{In Proposition~\ref{prop:enum}, ${n \brace p}$ is the number 
of partitions of $n$ objects in $p$ non-empty subsets (see e.g. \cite[p. 735--737]{FS09}).}.
\end{proposition}
\begin{proof}
Once the tree-structure of the binary tree is chosen (factor $2^{n-1}C_n$),
we partition the set of leaves into $p$ parts such that two leaves that belong to the same part are labelled
by the same variable. It gives the contribution ${n \brace p}$.
Then, we choose  to label each leaf by a positive or negative literal: contribution $2^n$.
The equivalence relationship states that a tree and the one obtained from it by replacing the positive literals
corresponding to a fixed variable by its negation (and conversely) are equivalent.
Thus, for each class we double-count the number of trees: correction $2^{-p}$.\hfill\qed
\end{proof}

Given a set $\mathcal{S}$ of equivalence classes of trees and $S_n$ the number of elements of $\mathcal{S}$
of size~$n$, we define the {\bf ratio} of $\mathcal{S}$ by $\mu_{n}(\mathcal{S}) = S_n / T_{n}$.
For a given Boolean function $f$, we denote by $T_n\langle f\rangle$ the number of equivalence classes
of trees of size~$n$ that compute a function of $\langle f\rangle$, and we define the {\bf probability} of $\langle f \rangle$ as
the ratio of $T_n\langle f\rangle$:
\[\IP_{n} \langle f\rangle = \frac{T_n\langle f\rangle }{T_{n}}.\]

One goal of this paper consists in studying the behaviour of the probabilities
 $(\IP_{n}\langle f\rangle)_{f\in\F}$ when the size $n$ of the trees tends to infinity.

\section{Results}\label{sec:results}

We state here our main result: the behaviour of $\IP_{n}\langle f\rangle$
for all fixed function $f\in \F$ in the framework of \ao trees.
Saying that $f$ is fixed means that its complexity (and its number of essential variables) is independent from $n$.

The main idea of this part is that \emph{a typical tree computing a Boolean function~$f$ is a minimal tree
of $f$ in which has been plugged a large tree that does not distort the function computed by the minimal tree.}

Since this main idea is identical in other framework (e.g. logic of implication~\cite{FGGG12}),
we are convinced that many recent results in quantitative logics could be translated in our new model too.

\begin{definition}
Let $\langle f\rangle$ be a fixed class of Boolean functions. We denote by $L\langle f \rangle$
(resp. $E\langle f \rangle$)  the common complexity (resp. number of essential variables)
of the functions of $\langle f\rangle$. The {\bf multiplicity} of the class $\langle f \rangle$, denoted by $R\langle f\rangle$, is the number $L\langle f \rangle - E\langle f \rangle$: it corresponds to the number of repetitions of variables in a minimal tree of $\langle f\rangle$.
\end{definition}

\begin{theorem}\label{thm:satis}
Let $(k_n)_{n\geq 1}$ be an increasing sequence of integers tending to $+\infty$ when $n$ tends to $+\infty$.
A random Catalan expression is satisfiable with probability tending to~$1$,
when the size of the expression tends to infinity.
\end{theorem}

\begin{theorem}\label{thm:theta}
Let $(k_n)_{n\geq 1}$ be an increasing sequence of integers tending to $+\infty$ when $n$ tends to $+\infty$.
There exists a sequence $(M_n)_{n\geq 1}$ such that $M_n\sim \frac{n}{\ln n}$ (when $n$ tends to $+\infty$) and such that, for all fixed equivalence class of Boolean functions $\langle f\rangle $, there exists a positive constant $\lambda_{\langle f\rangle}$ satisfying 
\begin{enumerate}[(i)]
\item if, for large enough $n$, $k_n \leq M_n$, then, asymptotically when $n$ tends to $+\infty$,
\[\mathbb{P}_n \langle f \rangle \sim \lambda_{\langle f\rangle}\cdot  \left( \frac{1}{k_{n+1}} \right)^{R\langle f \rangle+1} ;\]
\item if, for large enough $n$, $k_n \geq M_n$, then, asymptotically when $n$ tends to $+\infty$,
\[\mathbb{P}_n \langle f \rangle \sim \lambda_{\langle f\rangle} 	\cdot \left(\frac{\ln n}{n}\right)^{R\langle f \rangle+1}.\]
\end{enumerate}
\end{theorem}

Let us first remark that the constant $\lambda_{\langle f\rangle}$ is independent from $k_n$ (and from $n$).
Moreover both constant functions $\true$ and $\false$ are alone in their respective equivalence classes, and
their complexity is~$0$.\\

In the \emph{finite} context~\cite{CFGG04,kozik08}, each Boolean function is studied separately instead of being considered
among its equivalence class. 
We can translate the result obtained by Kozik in terms of equivalence classes by summing over all Boolean functions
belonging to a given equivalence class: remark that there are $\binom{k}{E(f)}2^{E(f)}$ functions in the equivalence class
of a given Boolean function $f$, therefore, the result of Kozik is equivalent to:
for all fixed Boolean function $\langle f\rangle$, asymptotically when $k$ tends to infinity,
\[\lim_{n\to+\infty} \IP_{n,k}\langle f\rangle = \Theta_{k \rightarrow \infty} \left(\frac1{k^{L(f) - E(f)+1}}\right)
 = \Theta_{k \rightarrow \infty}\left(\frac1{k^{R(f)+1}}\right).\]
Of course, the interchanging of both limits is not possible, but the finite model is not so far from being an extreme case
of our new model: the \emph{finite} context looks like a degenerate case of our model where there
exists an fixed integer $k$ such that $k_n=k$ for all $n\geq 1$. However, remark that we assume in the present paper that $k_n$ tends to $+\infty$ when $n$ tend to infinity: the case $k_n=k$ is thus not a particular case of our results.\\ 
 
Concerning the infinite context~\cite{GKZ07,GK12} $k_n=+\infty$, we already noticed that the cases such that $k_n$ is larger than
$n$ are equivalent to the model $k_n=n$, even if $k_n=+\infty$. 
Therefore, this infinite context is actually the extreme case $k_n=n$ of our model, and this particular case is thus fully treated in the present paper.

\section{Technical key points}\label{sec:technical}

In this section, we state the technical core of our results, and
we demonstrate how a threshold does appear according to the behaviour of $k_n$ as $n$ tends to infinity.

\subsection{Threshold induced by $k_n$'s behaviour}

\begin{definition}\label{df:B_n}
Let us define the following quantity: $B_{n,k_n} = \sum_{p=1}^{k_n} {n \brace p} 2^{-p}$.
The number $B_{n,k_n}$ quantitatively represents the labelling constraints of leaf-labelling by variables (cf. Proposition~\ref{prop:enum}).
\end{definition}
The following proposition, which can be seen as some particular case of Bonferroni inequalities allows to exhibit bounds on $B_{n,k_n}$.
\begin{proposition}[{\cite[Section~4.7]{Comtet74}}, or~\cite{Sibuya88} for a simpler proof]\label{fact:inegalite}
For all $n\geq 1$, for all $p\in\{1,\ldots,n\}$,
\[\frac{p^n}{p!} - \frac{(p-1)^n}{(p-1)!} \leq \left\{\begin{matrix}n\\p\end{matrix}\right\} \leq \frac{p^n}{p!}.\]
\end{proposition}
In view of these inequalities and of the expression of $B_{n,k_n}$ (cf. Definition~\ref{df:B_n}), it is natural to study the following sequences:
\begin{lemma}\label{lem:croissance}
Let $n$ be a positive integer.
\begin{enumerate}[(i)]
\item The sequence $(a_p^{(n)})_{p\in\{1,\ldots,n\}} = \left(\frac{p^n}{p!}2^{-p}\right)$ is unimodal. More precisely, there exists a integer $M_n$ such that $(a_p)_p$ is strictly increasing
on $\{1,2, \dots,  M_n\}$ and strictly decreasing on $\{M_n +1,\dots, n\}$.
\item Moreover, the sequence $(M_n)_n$ is increasing and asymptotically satisfies:
\[M_n \sim \frac{n}{\ln n}.\]
\end{enumerate}
\end{lemma}
The proof of this lemma is postponed to Appendix~\ref{app:technic_core}.

We are now ready, to understand the asymptotic behaviour of $B_{n,k_n}$:
\emph{roughly speaking, before the threshold ($k_n\leq M_n$), $B_{n,k_n}$ is equivalent to the sum of a few of its last terms,
and after $M_n$, it is equivalent to the sum of a few terms around $M_n$.}
\begin{lemma}\label{lem:u_n}
Let $(u_n)_{n\geq 1}$ be an increasing sequence such that $u_n\leq n$ for all integer $n\geq 1$ and $u_n$
tends to $+\infty$ when $n$ tends to $+\infty$.
\begin{enumerate}[(i)]
\item If, for all large enough $n$, $u_n\leq M_n$, then, for all sequences $(\delta_n)_{n\geq 1}$
such that $\delta_n = o(u_n)$ and $\frac{u_n\sqrt{\ln u_n}}{n} = o(\delta_n)$, we have, asymptotically when $n$ tends to $+\infty$,
\begin{equation}\label{u_n_petit}
B_{n,u_n} = \Theta\left(\sum_{p=u_n-\delta_n}^{u_n}\frac{p^n}{p!}2^{-p}\right).
\end{equation}
\item If, for large enough $n$, $u_n\geq M_n$, then, for all sequences $(\delta_n)_{n\geq 1}$ such that $\delta_n = o(u_n)$
 and $\frac{u_n\sqrt{\ln u_n}}{n} = o(\delta_n)$, for all sequences $(\eta_n)_{n\geq 1}$
 such that $\eta_n = o(M_n)$, $\lim_{n\to+\infty}\frac{\eta_n^2}{M_n} = +\infty$ and $\sqrt{M_n\ln (u_n-M_n)}=o(\eta_n)$,
 we have, asymptotically when $n$ tends to $+\infty$,
\begin{equation}\label{u_n_grand}
B_{n,u_n} = \Theta\left(\sum_{p=M_n-\delta_n}^{\min\{M_n + \eta_n, u_n\}} \frac{p^n}{p!}2^{-p}\right).
\end{equation}
\end{enumerate}
\end{lemma}
This lemma is proved in Appendix~\ref{app:technic_core}
and allows us to deduce the following results on the behaviour of $B_{n,k_n}$, when $n$ tends to $+\infty$:
\begin{lemma}\label{lem:kpetit}
Let $(k_n)_{n\geq 1}$ be a sequence of integers that tends to $+\infty$ when $n$ tends to $+\infty$.
Let us assume that $k_n \leq M_n$ for large enough $n$, then, asymptotically when $n$ tends to infinity,
\[\frac{B_{n,k_{n+1}}}{B_{n+1,k_{n+1}}} = \Theta\left(\frac1{k_{n+1}}\right).\]
\end{lemma}

\begin{lemma}\label{lem:kgrand}
Let $(k_n)_{n\geq 1}$ be a sequence of integers that tends to $+\infty$ when $n$ tends to $+\infty$.
Let us assume that $k_{n} \geq M_n$ for large enough $n$, then, asymptotically when $n$ tends to infinity,
\[\frac{B_{n,k_{n+1}}}{B_{n+1,k_{n+1}}} = \Theta\left(\frac{\ln n}{n}\right).\]
\end{lemma}

\begin{definition}
Let the fraction $\rat_n$ be the quantitative evolution of the leaf-labelling constraints
from trees of size $n-1$ to size~$n$:
$\rat_n = B_{n-1,k_{n}}/B_{n,k_{n}}$.\\
Its asymptotic behaviour is quantified by Lemmas~\ref{lem:kpetit} and~\ref{lem:kgrand}.
\end{definition}

\subsection{Adjustment of Kozik's pattern language theory}

In 2008, Kozik~\cite{kozik08} introduced a quite effective way to study Boolean trees: 
he defined a notion of pattern that permits to easily classify and count large trees according to some constraints
on their structures.
Kozik applied this pattern theory to study \ao trees with a finite number of variables.
This theory has then been extended to different models of Boolean trees (see for example paper~\cite{GGKM13}).

We adapt the definitions of patterns to our new model and then we extend results
of Kozik's paper.
\begin{definition}
\begin{enumerate}[(i)]
\item A {\bf pattern} is a binary tree with internal nodes labelled by
$\land$ or $\lor$ and with external nodes labelled by~$\bullet$ or~$\boxempty$.
Leaves labelled by~$\bullet$ are called {\bf pattern leaves} and leaves labelled by $\boxempty$
are called {\bf placeholders}. A {\bf pattern language} is a set of patterns
\item Given a pattern language $L$ and a family of trees $\M$, we denote by $L[\M]$ the family of all trees
obtained by replacing every placeholder in an element from $L$ by a tree from $\M$. 
\item We say that $L$ is {\bf unambiguous} if, and only if, for any family $\mathcal{M}$ of trees, any tree of $L[\mathcal M]$
can be built from a unique pattern from $L$ in which has been plugged trees from $\mathcal{M}$.
\end{enumerate}
\end{definition}
The generating function of a pattern language $L$ is $\ell(x,y) = \sum_{d, p} L(d,p)x^dy^p$,
where $L(d,p)$ is the number of elements of $L$ with $d$ pattern leaves and $p$ placeholders.

\begin{definition}
We define the composition of two pattern languages $L[P]$ as the pattern language
of trees which are obtained by replacing every placeholder of a tree from $L$ by a tree from $P$. 
\end{definition}

\begin{definition}
A pattern language $L$ is {\bf sub-critical} for a family $\M$ if the generating function $m(z)$ of $\M$
has a square-root singularity $\tau$, and if $\ell(x,y)$ is analytic in some set
$\{(x,y) : |x|\leq \tau+\varepsilon, |y|\leq m(\tau)+\varepsilon\}$ for some positive $\varepsilon$. 
\end{definition}

\begin{definition}
Let $L$ be a pattern language, $\mathcal M$ be a family of trees and $\Gamma$ a subset of $\{x_i\}_{i\geq 1}$,
whose cardinality does not depend on $n$. Given an element of $L[\M]$, 
\begin{enumerate}[(i)]
\item the number of its $L$-{\bf repetitions} is the number of its $L$-pattern leaves minus the number of different variables that appear in the labelling of its $L$-pattern leaves.
\item the number of its $(L,\Gamma)$-{\bf restrictions} is the number of its $L$-pattern leaves that are labelled by variables from $\Gamma$, plus the number of its $L$-repetitions.
\end{enumerate}
\end{definition}

\begin{definition}
Let $\I$ be the family of the trees with internal nodes labelled by a
 connective and leaves without labelling, i.e. the family of tree-structures.
\end{definition}
The generating function of $\I$ satisfies $I(z) = z+2I(z)^2$, that implies $I(z) = (1-\sqrt{1-8z})/4$ and
thus its dominant singularity is $\nicefrac18$.\\

We can, for example, define the unambiguous pattern language $N$ by induction as follows: $N = \bullet | N\lor N | N\land \boxempty$, meaning that a pattern from  $N$ is either a single pattern leaf, or a tree rooted by $\lor$ whose two subtrees are patterns from $N$, or a tree rooted by $\land$ whose left subtree is a pattern from $N$ and whose right subtree is a placeholder. Its generating function verifies, by symbolic arguments, $n(x,y) = x + n(x,y)^2 + yn(x,y)$ and is equal to $n(x,y) = \frac12(1-y-\sqrt{(1-y)^2-4x})$. It is thus subcritical for $\mathcal{I}$.

On the left-hand side of~Fig.~\ref{fig:true}, we have depicted a Boolean tree that computes the constant function $\true$. It has $5$ $N$-pattern leaves, $1$ $N$-repetition and $2$ $(N,\{x_2\})$-restrictions.

\begin{figure}[t]
\centering
\begin{tabular}{c|c}
\begin{tikzpicture}[style={level distance=1cm},level 1/.style={sibling distance=3cm},level 2/.style={sibling distance=2cm}, scale=0.8]
 \node [circle] (z){$\vee$}
	child {node [circle] (x) {$\vee$}
		child{node [circle] (a) {$\vee$}
				child {node [circle] (y) {$x_1$}}
				child {node [circle] (w) {$\vee$}
					child {node [circle] (b) {$\bar{x}_1$}}
					child {node [circle] (c) {$x_2$}}
					}
				}
		child {node [circle] (d) {$x_3$}}
		}
	child {node [circle] (g) {$\wedge$}
 				child {node [circle] (e) {$x_4$}}
 				child {node [circle] (f) {$x_1$}} 
				}
  ;
\end{tikzpicture}
&
\begin{tikzpicture}[style={level distance=1cm},level/.style={sibling distance=40mm/#1}, scale=0.8]
 \node [circle] (z){$\vee$}
 child {node [circle] (a){$\vee$}
   	child{node [circle] (b) {$\vee$} 
 		child{node [circle] (c) {$\cdots$} }
 		child{node [circle] (d) {$\vee$} 
 			child{node [circle] (x) {$x$}}
 			child{node [circle] (y) {$\cdots$}}
 			}
 		}
 	child{node [circle] (e) {$\cdots$} }
 	}
 child {node [circle] (f) {$\vee$}
 	child{node [circle] (g) {$\cdots$} }
 	child{node [circle] (h) {$\vee$}
 		child{node [circle] (i) {$\bar{x}$} }
 		child{node [circle] (l) {$\cdots$} }
 		}
 	}
 ;
\end{tikzpicture}
\end{tabular}
\caption{Left: a Boolean tree that computes the function $\true$.
Right: a simple tautology.}
\label{fig:true}
\end{figure}
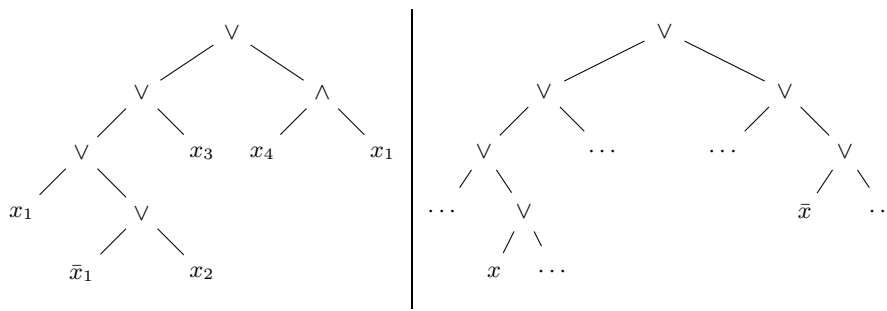

The following key-lemma is a generalization of Kozik's one~\cite[Lemma 3.8]{kozik08}: 

\begin{lemma}\label{lem:Jakub}
Let $L$ be an unambiguous pattern, and $\T$ the families of \ao trees.
Let $T^{[r]}_{n}$ (resp. $T^{[\geq r]}_{n}$) be the number
 of labelled (with at most $k_n$ variables) trees of $L[\T]$
of size $n$ and with $r$ $L$-repetitions (resp. at least $r$ L-repetitions).
We assume that $L$ is sub-critical for the family $\I$ of the unlabelled-leaves trees.
Then, asymptotically when $n$ tends to infinity,
\[\frac{T^{[r]}_{n}}{T_{n}} = \mathcal{O}\left(\rat_n^r\right) \hspace*{1cm}\text{ and }\hspace*{1cm}
\frac{T^{[\geq r]}_{n}}{T_{n}} = \mathcal{O}\left(\rat_n^r\right).\]
\end{lemma}
\begin{proof}
The number of labelled trees of $L[\T]$ of size $n$ and with at least $r$ $L$-repetitions is given by:
\[T_{n}^{[\geq r]} = \sum_{d=r+1}^{n} I_n(d) Lab(n,k_n,d,r),\]
where $I_n(d)$ is the number of tree-structures with $d$ L-pattern leaves and
the number $Lab(n,k_n,d,r)$ corresponds to the number of leaf-labellings of these trees giving at least $r$ $L$-repetitions.
The following enumeration contains some double-counting and we therefore get an upper bound:
\[Lab(n,k_n,d,r) \leq 2^n \cdot  \sum_{j=1}^{r}\binom{d}{r+j} {r+j \brace j} B_{n-r-j+1,k_n}.\]
The factor $2^n$ corresponds to the polarity of each leaf (the variable labelling it is either negated or not);
the index $j$ stands for the number of different variables involved in the $r$ repetitions;
the binomial factor chooses the pattern leaves that are involved in the $r$ repetitions;
the Stirling number partition splits $r+j$ leaves into $j$ parts;
finally, the factor $B_{n-r-j+1,k_n}$ chooses which variable is assigned to each class of leaves.
Therefore,
\[T_{n}^{[\geq r]} 
\leq 2^n\cdot B_{n-r,k_n} \sum_{j=1}^{r}  {r+j \brace j} \sum_{d=r+j}^{n} I_n(d) \binom{d}{r+j}.\]
Let $\ell(x,y)$ be the generating function of the pattern $L$. Then, for all $p\geq 0$,
\[\frac{z^p}{p!} \frac{\partial^p \ell}{\partial x^p}(z,I(z)) = \sum_{n=1}^{\infty}\sum_{d=1}^{\infty} I_n(d)\binom{d}{p}z^n.\]
Thus, 
\[\frac{T_{n}^{[\geq r]}}{T_{n, k_n}} \leq 
\frac{B_{n-r,k_n}}{B_{n,k_n}}\sum_{j=1}^r {r+j \brace j} \frac{[z^n]z^{r+j}\frac{\partial^{r+j} \ell}{\partial x^{r+j}}(z,I(z))}{[z^n]I(z)}.\]
Since $z^{r+j}\frac{\partial^{r+j} \ell}{\partial x^{r+j}}(z,I(z))$ and $I(z)$
have the same singularity because of the sub-criticality of the pattern $L$ according to $\I$,
the previous sum is constant when $n$ tends to infinity and so
we conclude:
\[\frac{T_{n}^{[r]}}{T_{n}} 
\leq \frac{T_{n}^{[\geq r]}}{T_{n}} 
= O\left(\frac{B_{n-r, k_{n}}}{B_{n, k_n}} \right)
= O\left(\rat_n^r\right).\]
\indent \hfill\qed
\end{proof}

\section{Behaviour of the probability distribution}\label{sec:proba}

Once we have adapted the pattern theory to our model,
we are ready to quantitatively study it.
A first step is to understand the asymptotic behaviour of $\IP_{n}\langle \true\rangle$.
It is indeed natural to focus on this ``simple'' function before considering a general class $\langle f \rangle$;
and moreover, it happens to be essential for the continuation of the study.
In addition, the methods used to study tautologies (mainly pattern theory)
will also be the core of the proof for a general equivalence class.
We prove in this section the main Theorem~\ref{thm:theta} for both classes $\langle \true \rangle$ and $\langle \false\rangle$
of complexity zero, using the
duality of both connectives $\land$ and $\lor$ and both positive and negative literals.
The main ideas of the proof for a general equivalence class will be detailed in Section~\ref{sec:main},
but the details will be postponed into Appendix~\ref{sec:general_class}.

\subsection{Tautologies}

Let us recall that a {\bf tautology} is a tree that represents the Boolean function $\true$.
Let us consider the family $\A$ of tautologies.
In this part, we prove that the probability of $\langle \true \rangle$ is
equivalent to the ratio of a simple subset of tautologies.

\begin{definition}[cf. right-hand side of Fig.~\ref{fig:true}]
A {\bf simple tautology} is an \ao tree that contains two leaves
labelled by a variable $x$ and its negation $\bar{x}$ and
such that all internal nodes from the root to both leaves are labelled by $\lor$-connectives.
We denote by $ST$ the family of simple tautologies.
\end{definition}
\begin{proposition}\label{prop:tauto}
The ratio of simple tautologies verifies
\[\mu_{n}(ST) = \frac{ST_{n}}{T_{n}} \sim \frac34 \rat_n, \text{ when $n$ tends to infinity.}\]
Moreover, asymptotically when $n$ tends to infinity, almost all tautologies are simple tautologies.
\end{proposition}
The detailed proof is given in Appendix~\ref{sec:tauto}.

The latter proposition gives us for free the proof of Theorem~\ref{thm:satis}. 
In fact, both dualities between the two connectives and positive and negative literals
transform expression computing $\true$ to expressions computing $\false$, 
which implies $\IP_n\langle \false\rangle = \nicefrac{3}{4}\cdot\rat_n$.
Moreover, the only expressions that are not satisfiable compute the function $\false$ and 
$\IP_n\langle \false\rangle = \nicefrac{3}{4}\cdot\rat_n$ tends to $0$ as $n$ tends to infinity, 
which proves Theorem~\ref{thm:satis}.

\subsection{Probability of a general class of functions}\label{sec:main}

With similar arguments than those used for tautologies,
we prove that the probability of the class of projections (i.e. $(x_i)_{i\geq 1} \mapsto x_j$)
is equivalent to $\nicefrac58 \cdot \rat_n$. The proof is detailed in Appendix~\ref{sec:projections}.

Let us turn now to the general result:
the behaviour of $\IP_n\langle f\rangle$ for all fixed $f\in \F$.
The main idea of this part is that, roughly speaking,
\emph{a typical tree computing a Boolean function in $\langle f\rangle$ is a minimal tree
of $\langle f\rangle$ in which has been plugged a single large tree}.
Here we give the main ideas of the proof of Theorem~\ref{thm:theta},
the complete proof is given in Appendix~\ref{sec:general_class}.

\begin{proof}[sketch]
For a given class of Boolean functions $\langle f\rangle$ our goal is to 
obtain an asymptotic equivalent to $\IP_n\langle f\rangle$.
\begin{itemize}
\item We first define several notions of \emph{expansions} of a tree: the idea is to replace
in a tree, a subtree $S$
by $T \land S$, where $T$ is chosen such that the expanded tree still computes the same function.

\item The ratio of minimal trees of $\langle f\rangle$ expanded once is of the order of $\rat_n^{R(f)+1}$.

\item The ratio of trees computing a function from $\langle f\rangle$ is equivalent to the ratio of minimal trees expanded once.
\end{itemize}
The most technical part of the proof is the last one, because we need a precise upper bound of $\IP_n\langle f\rangle$.
But the ideas are more or less the same as those developed for the class $\langle \true\rangle$.\hfill\qed
\end{proof}

\section{Conclusion}

We focus on the logical context of \ao connectives
because of the richness of this logical system (normal forms, functional completeness).
However the implicational logical system (e.g.~\cite{FGGG12,GK12}) could also be studied in this new context
and we deeply believe the general behaviour to be identical.
Indeed, the key idea is that \emph{each repetition induces a factor $\rat_n$},
and this remains true in all those models -- although pattern theory does not adapt to every model, e.g. models with \emph{implication}.
Extending our results to these models would give nice unifications
of the known results of the literature:
papers~\cite{kozik08,FGGG12,GK12} and~\cite{GGKM12,GGKM13}.

With our new model, we can now relate the large number of results
that have been obtained during the last decade on quantitative logics
to problems about satisfiability.
Our Catalan model of expressions behaves differently since,
asymptotically, almost all expressions are satisfiable, whatever the ratio between the number of variables
and the size of expressions.

To conclude, the specific form of expressions in $k-SAT$ problems deeply bias the probability distribution on Boolean functions.

\bibliography{boolean}

\newpage
\appendix

\section{Proofs of the technical core}\label{app:technic_core}

\begin{proof}[of Lemma~\ref{lem:croissance}]
{\it (i)} 
Let us prove that the sequence $(a_p)_{1\leq p\leq n}$ is log-concave, 
i.e. that the sequence $\left(\frac{a_{p+1}}{a_p}\right)_{1\leq p\leq n-1}$ is decreasing.  
Let $p$ be an integer in $\{1,\dots, n-1\}$. By Definition of $a_p^{(n)}$:
\[\frac{a_{p+1}^{(n)}}{a_p^{(n)}} = \left(\frac{p+1}{p}\right)^n\frac1{2(p+1)},\]
and consequently, for all $n\geq 0$,
\[\frac{a_{p+1}^{(n)}}{a_p^{(n)}} > 1 \iff
 n\ln\left(\frac{p+1}{p}\right)-\ln(2(p+1)) > 0.\]
The function $\phi \,:\, p\mapsto n\ln\left(\frac{p+1}{p}\right)-\ln(2(p+1))$ is strictly decreasing.
Since $\phi(1)$ tends to $+\infty$ and $\phi(n-1)$ tends to $-\infty$ when $n$ tends to infinity, 
there exists a unique $M_n$ such that $(a_p)$ is strictly increasing on $\{1,\dots,M_n\}$
and strictly decreasing on $\{M_n+1,\dots, n\}$.\\
{\it (ii)} Let us denote by $x_n$ the single solution of equation:
\begin{equation}\label{eq:max}
\left(\frac{x+1}{x}\right)^n\frac1{2(x+1)} = 1.
\end{equation}
First remark that the sequence $(x_n)_{n\geq 1}$ is increasing. We indeed know: $\phi_n(x_n) = 0$ and $\phi_{n+1}(x_{n+1}) = 0$,
which implies that $\phi_n(x_{n+1}) = -\ln \left(1+\frac{1}{x_{n+1}}\right) < 0$.
Therefore, since $(\phi_n)_{n\geq 1}$ is decreasing, we have that $x_{n+1}\geq x_n$, for all large enough $n$.
Therefore, the sequence $(M_n)_{n\geq 1}$ is asymptotically increasing.

Since, asymptotically when $n$ tends to infinity,
\[\left(\frac{\frac{n}{\ln n}+1}{\frac{n}{\ln n}}\right)^n\frac1{2(\frac{n}{\ln n}+1)} \sim \frac{\ln n}{2},\]
we have that $n / \ln n\leq x_n$ and therefore,
$x_n$ tends to infinity. Thus,
Equation~(\ref{eq:max}) evaluated in $x_n$ is equivalent to
\begin{equation}\label{eq:M_n}
n\ln\left(1+\frac1{x_n}\right) = \ln 2 +\ln (x_n+1),
\end{equation}
which implies $x_n\ln x_n \sim n$ when $n$ tends to infinity.
We easily deduce from this asymptotic relation that
$\ln x_n \sim \ln n$ and that $x_n \sim \frac{n}{\ln n}$ when $n$ tends to infinity.
Since $M_n = \lfloor x_n \rfloor$, we conclude that $M_n \sim \nicefrac{n}{\ln n}$ when $n$ tends to infinity.
\hfill\qed
\end{proof}

In view of Proposition~\ref{fact:inegalite}, we have the following bounds:
\begin{equation}\label{eq:bounds}
\frac12 \cdot \sum_{p=1}^{u_n-1} \frac{p^n}{p!\ 2^p} + \frac{u_n^n}{u_n!\ 2^{u_n}} \leq B_{n,u_n} 
\leq \sum_{p=1}^{u_n} \frac{p^n}{p!\ 2^p}.
\end{equation}

\begin{proof}[of Lemma~\ref{lem:u_n} {\it (i)}]
Via Proposition~\ref{fact:inegalite}, we can bound $B_{n,u_n}$ : for all $n\geq 1$,
\begin{equation}\label{eq:infini_bounds}
\frac12 \cdot \sum_{p=1}^{u_n-1} \frac{p^n}{p!\ 2^p} + \frac{u_n^n}{u_n!\ 2^{u_n}} \leq B_{n,u_n} 
\leq \sum_{p=1}^{u_n} \frac{p^n}{p!\ 2^p}.
\end{equation}

Let us assume that $u_n\leq M_n$ for all large enough $n$, and let us prove that the two bounds of Equation~\eqref{eq:infini_bounds} are of the same asymptotic order when $n$ tends to $+\infty$.

Remark that for all integer $N\geq 1$, $S_N=\sum_{p=1}^{N} a_p^{(n)}$ : Equation~\eqref{eq:infini_bounds} implies
\[\frac{1}{2}S_{u_n} \leq B_{n,u_n} \leq S_{u_n} .\]
Let us split the sum $S_{u_n}$ into two sums: the last $\delta_n$ summands, and the rest.
\[S_{u_n} = S_{u_n-\delta_n-1} + \sum_{p=u_n-\delta_n}^{u_n} a_{p}.\]
By assumption, $\delta_n = o(u_n)$ and we therefore can choose $n$ large enough such that $u_n>\delta_n$.
Let us prove that $S_{u_n-\delta_n-1}$ is negligible in front of $a_{u_n}$, and thus in front of $\sum_{p=u_n-\delta_n}^{u_n} a_p$. Recall that $(a_p)_{p\geq 1}$ is increasing on $\{1,\ldots, M_n\}$, which implies
\[S_{u_n-\delta_n-1} \leq u_n a_{u_n-\delta_n}.\]
For all large enough $n$, via Stirling formula,
\begin{align*}
\frac{a_{u_n-\delta_n}}{a_{u_n}}
& = 2^{\delta_n}\left(\frac{u_n-\delta_n}{u_n}\right)^n \frac{u_n !}{(u_n-\delta_n)!}\\
& = \left(\frac{2u_n}{\e}\right)^{\delta_n}\left(\frac{u_n-\delta_n}{u_n}\right)^{n-u_n+\delta_n-\nicefrac12}(1+o(1))\\
& = \exp\left[\delta_n\ln 2 -\delta_n+\delta_n u_n + (n-u_n+\delta_n-\nicefrac12)\ln\left(1-\frac{\delta_n}{u_n}\right)+o(1)\right].
\end{align*}
Since $\delta_n = o(u_n)$, we have $\ln\left(1-\frac{\delta_n}{u_n}\right) = -\frac{\delta_n}{u_n} - \frac{\delta_n^2}{2u_n^2}$, and
\begin{align*}
\frac{a_{u_n-\delta_n}}{a_{u_n}}
& = \exp\left[\delta_n\ln 2 -\delta_n+\delta_n u_n - \frac{n\delta_n}{u_n} + \delta_n - \frac{n\delta_n^2}{2u_n^2}+ \mathcal{O}\left(\frac{n\delta_n^2}{u_n^2}\right)\right]\\
& = \exp\left[\delta_n\ln 2 +\delta_n u_n - \frac{n\delta_n}{u_n} - \frac{n\delta_n^2}{2u_n^2}+ o\left(\frac{n\delta_n^2}{u_n^2}\right)\right],
\end{align*}
because, by hypothesis, $\ln u_n = o\left(\frac{\delta_n^2}{u_n}\right)$, which implies $\frac{n\delta_n^2}{u_n^2} = \Omega(\ln u_n)$.
Since $u_n\leq M_n$, and in view of Equation~\eqref{eq:M_n}, $\frac{n}{M_n} \geq \ln 2 + \ln M_n$. Therefore,
\begin{align*}
\frac{a_{u_n-\delta_n}}{a_{u_n}}
&  \leq \exp\left[\delta_n\ln 2 +\delta_n M_n - \frac{n\delta_n}{M_n} - \frac{n\delta_n^2}{2u_n^2}+ o\left(\frac{n\delta_n^2}{u_n^2}\right)\right]\\
& \leq \exp\left[- \frac{n\delta_n^2}{2u_n^2}+ o\left(\frac{n\delta_n^2}{u_n^2}\right)\right].
\end{align*}
Since $\frac{n\delta_n^2}{u_n^2} = \Omega(\ln u_n)$, we can conclude that
\[\frac{S_{u_n-\delta_n-1}}{a_{u_n}} \leq u_n \frac{a_{u_n-\delta_n}}{a_{u_n}} \leq \exp\left[\ln u_n - \frac{n\delta_n^2}{2u_n^2}+ o\left(\frac{n\delta_n^2}{u_n^2}\right)\right] = o(1).\]
It implies $S_{u_n} \sim \sum_{p=u_n-\delta_n}^{u_n} a_{p}$, which ends the proof.
\indent\hfill\qed
\end{proof}

\begin{proof}[of Lemma~\ref{lem:u_n},  {\it (ii)}]
Assume that $u_n\geq M_n$ for all large enough $n$.
Let us split the sums of the two bounds of Equation~\eqref{eq:infini_bounds} into three parts: the first from index 1 to $M_n-\delta_n-1$, the second from index $M_n-\delta_n$ to $M_n + \eta_n$, and the third from index $M_n + \eta_n+1$ to $u_n$. Remark that, if $u_n\leq M_n+\eta_n$, then the third part equals zero and the second part is truncated:
\[S_{u_n} = S_{M_n-\delta_n-1} + \sum_{p=M_n-\delta_n}^{M_n+\eta_n} a_p + \sum_{p=M_n+\eta_n+1}^{u_n} a_p.\]

By arguments similar to those developped in the proof of assertion {\it (i)}, we can prove that $S_{M_n-\delta_n-1}$ is negligible in front of $a_{M_n}$, and thus in front of $\sum_{p=M_n-\delta_n}^{M_n+\eta_n} a_p$. Therefore, if $u_n\geq M_n+\eta_n$, assertion {\it (ii)} is proved.
Let us now assume that $u_n\geq M_n+\eta_n+1$:to end the proof, we have to prove that $\sum_{p=M_N+\eta_n+1}^{u_n} a_p$ is negligible in front of $a_{M_n}$, and thus in front of $\sum_{p=M_n-\delta_n}^{M_n+\eta_n} a_p$.

In view of Lemma~\ref{lem:croissance}, we have
\[\sum_{p=M_n+\eta_n+1}^{u_n} a_p \leq (u_n-M_n-\eta_n) a_{M_n+\eta_n}.\]
Via Stirling formula,
\begin{align*}
\frac{a_{M_n+\eta_n}}{a_{M_n}}
& = 2^{-\eta_n} \left(\frac{M_n+\eta_n}{M_n}\right)^n \frac{M_n!}{(M_n+\eta_n)!}\\
& = \left(\frac{2(M_n+\eta_n)}{\e}\right)^{-\eta_n}   \left(\frac{M_n+\eta_n}{M_n}\right)^{n-M_n-\nicefrac12} (1+o(1))\\
& = \exp\left[-\eta_n\ln 2 + \eta_n - \eta_n\ln(M_n+\eta_n) + (n-M_n-\nicefrac12)\ln \left(1+\frac{\eta_n}{M_n}\right)+o(1)\right].
\end{align*}
Since, by hypothesis, $\ln \left(1+\frac{\eta_n}{M_n}\right) \leq \frac{\eta_n}{M_n}$ and $\frac{\eta_n}{M_n} = o(1)$, we have
\begin{align*}
\frac{a_{M_n+\eta_n}}{a_{M_n}} 
& \leq \exp\left[-\eta_n\ln 2 + \eta_n - \eta_n\ln (M_n+\eta_n) + \frac{\eta_n}{M_n}(n-M_n-\nicefrac12)+o(1)\right]\\
& = \exp\left[-\eta_n\ln 2 + \eta_n - \eta_n\ln (M_n+\eta_n) + \frac{n\eta_n}{M_n}-\eta_n + o(1)\right]\\
& = \exp\left[-\eta_n\ln 2 - \eta_n\ln (M_n+\eta_n) + \frac{n\eta_n}{M_n} + o(1)\right]\\
& = \exp\left[-\eta_n\ln 2 - \eta_n\ln M_n - \eta_n\ln\left(1+\frac{\eta_n}{M_n}\right) + \frac{n\eta_n}{M_n} + o(1)\right]\\
& = \exp\left[-\eta_n\ln 2 - \eta_n\ln M_n - \frac{\eta^2_n}{M_n} + \frac{n\eta_n}{M_n} + \mathcal{O}\left(\frac{\eta_n^3}{M^2_n}\right)\right]
\end{align*}
Since $M_n = \lfloor x_n \rfloor$, we have
\[n\ln\left(1+\frac1{x_n}\right) = n\left(\frac1{M_n} - \frac{1}{2M_n^2} + \mathcal{O}\left(\frac{1}{M_n^3}\right)\right),\]
and
\[\ln 2 + \ln (x_n+1) = \ln 2 + \ln M_n + \mathcal{O}\left(\frac1{M_n}\right).\]
In view of Equation~\eqref{eq:M_n}, it implies
\[\frac{n}{M_n} = \ln 2 + \ln M_n + \frac{n}{2M_n^2} + \mathcal{O}\left(\frac{n}{M_n^3}\right) + \mathcal{O}\left(\frac1{M_n}\right) =  \ln 2 + \ln M_n + \frac{n}{2M_n^2} + \mathcal{O}\left(\frac{n}{M_n^3}\right)\]
since $\frac1{M_n} = o(\frac{n}{M_n^3})$. Thus
\begin{align*}
\frac{a_{M_n+\eta_n}}{a_{M_n}} 
& \leq \exp\left[ - \frac{\eta_n^2}{M_n} + \mathcal{O}\left(\frac{\eta_n^3}{M_n^2}\right) + \mathcal{O}\left(\frac{n\eta_n}{M^3_n}\right)\right]\\
& = \exp\left[- \frac{\eta_n^2}{M_n} + o\left(\frac{\eta_n^2}{M_n}\right)\right],
\end{align*}
because, by hypothesis,  $\frac{\eta_n^2}{M_n}$ tends to $+\infty$ when $n$ tends to $+\infty$. 
We thus get
\[\frac{\sum_{p=M_n+\eta_n+1}^{u_n} a_p}{a_{M_n}} 
\leq (u_n-M_n-\eta_n) \frac{a_{M_n+\eta_n}}{a_{M_n}}
\leq \exp\left[\ln (u_n-M_n) - \frac{\eta_n^2}{M_n} + o\left(\frac{\eta_n^2}{M_n}\right)\right] = o(1)
\]
since, by hypothesis, $\ln (u_n-M_n) = o\left(\frac{\eta_n^2}{M_n}\right)$.
Therefore, asymptotically when $n$ tends to $+\infty$,
\[S_{u_n}\sim \sum_{p=M_n-\delta_n}^{M_n+\eta_n} a_p,\]
which concludes the proof.
\indent \hfill\qed
\end{proof}

\begin{proof}[of Lemma~\ref{lem:kpetit}]
{\bf Let us first assume that $\bs{k_{n+1}\leq M_n}$.}
Let $(\delta_n)_{n\geq 1}$ an integer-valued sequence such that $\delta_n = o(k_{n+1})$ and $\frac{k_{n+1}\sqrt{\ln k_{n+1}}}{n} = o(\delta_n)$ when $n$ tends to $+\infty$.
Lemma~\ref{lem:u_n} applied to $u_n = k_{n+1}$ gives, asymptotically when $n$ tends to infinity,
\[B_{n,k_{n+1}} = \Theta\left(\sum_{p=k_{n+1}-\delta_n}^{k_{n+1}} a_p^{(n)}\right).\]
Moreover, since $k_{n+1}\leq M_{n+1}$, and since the sequence $(\delta_{n-1})_{n\geq 1}$ verifies $\delta_{n-1}=o(k_n)$ and $\frac{k_{n}\sqrt{\ln k_{n}}}{n} = o(\delta_{n-1})$, applying Lemma~\ref{lem:u_n} to the sequence $u_n = k_n$ gives us, asymptotically when $n$ tends to infinity,
\[B_{n,k_{n}} = \Theta\left(\sum_{p=k_{n}-\delta_{n-1}}^{k_{n}} a_p^{(n)}\right),\]
which implies
\[B_{n+1,k_{n+1}} = \Theta\left(\sum_{p=k_{n+1}-\delta_{n}}^{k_{n+1}} a_p^{(n+1)}\right).\]
Therefore,
\[\rat_{n+1} := \frac{B_{n,k_{n+1}}}{B_{n+1,k_{n+1}}} 
= \Theta\left(\frac{\sum_{p=k_{n+1}-\delta_n}^{k_{n+1}} a_p^{(n)}}{\sum_{p=k_{n+1}-\delta_{n}}^{k_{n+1}} a_p^{(n+1)}}\right).
\]
We have
\begin{eqnarray*}
(k_{n+1}-\delta_n)\sum_{p=k_{n+1}-\delta_{n}}^{k_{n+1}} a_p^{(n)}
&\leq& \sum_{p=k_{n+1}-\delta_{n}}^{k_{n+1}} p a_p^{(n)}
= \sum_{p=k_{n+1}-\delta_{n}}^{k_{n+1}} a_p^{(n+1)}\\
&=& \sum_{p=k_{n+1}-\delta_{n}}^{k_{n+1}} p a_p^{(n)}
\leq k_{n+1}\sum_{p=k_{n+1}-\delta_{n}}^{k_{n+1}} a_p^{(n)},
\end{eqnarray*}
which implies
\[\rat_{n+1} = \frac{B_{n,k_{n+1}}}{B_{n+1,k_{n+1}}} = \Theta\left(\frac1{k_{n+1}}\right).\]

{\bf Now, let us assume $\bs{M_{n+1} \geq k_{n+1} > M_n}$}. Let $(\delta_n)_{n\geq 1}$ be an integer-valued sequence such that $\delta_n = o(k_{n+1})$ and $\frac{k_{n+1}\sqrt{\ln k_{n+1}}}{n+1} = o(\delta_n)$. Let $(\eta_n)_{n\geq 1}$ be an integer-valued sequence such that $\eta_n = o(M_n)$, $\lim_{n\to+\infty}\frac{\eta_n^2}{M_n} = +\infty$ and $\sqrt{M_n\ln (u_n-M_n)}=o(\eta_n)$.
Applying Lemma~\ref{lem:u_n} {\it (ii)} to the sequence $u_n = k_n$, we obtain
\[B_{n,k_{n+1}} = \Theta\left(\sum_{p=M_n-\delta_{n}}^{\min\{M_n+\eta_n,k_{n+1}\}} a_p^{(n)}\right).\]
Moreover, since $\delta_{n-1} = o(k_{n})$ and $\frac{k_n\sqrt{\ln k_n}}{n} = o(\delta_{n-1})$, via Lemma~\ref{lem:u_n} {\it (i)},applied to the sequence $u_n=k_n$,
\[B_{n+1,k_{n+1}} = \Theta\left(\sum_{p=k_{n+1}-\delta_n}^{k_{n+1}} a_p^{(n+1)}\right).\]
Let us remark, as above, that
\[(k_{n+1}-\delta_n) \sum_{p=k_{n+1}-\delta_n}^{k_{n+1}} a_p^{(n)} \leq B_{n+1,k_{n+1}}\leq k_{n+1}\sum_{p=k_{n+1}-\delta_n}^{k_{n+1}} a_p^{(n+1)}.\]
Moreover, since $k_{n+1}\geq M_n$, via similar arguments as thos developped to prove Lemma~\ref{lem:u_n} {\it (i)},
\[\sum_{p=k_{n+1}-\delta_n}^{k_{n+1}} a_p^{(n)} \sim \sum_{p=k_{n+1}-\delta_n}^{\min\{k_{n+1},M_n+\eta_n\}} a_p^{(n)} \sim B_{n,k_{n+1}}.\]
Therefore, since $\delta_n = o(k_{n+1})$, we get
\[\rat_{n+1} :=\frac{B_{n,k_{n+1}}}{B_{n+1,k_{n+1}}}
= \mathcal{O}\left(\frac1{k_{n+1}-\delta_n}\right) = \mathcal{O}\left(\frac1{k_{n+1}}\right).\]
Similar arguments lead to,
\[\rat_{n+1} = \Omega\left(\frac1{k_{n+1}}\right),\]
which concludes the proof.
\indent\hfill\qed
\end{proof}

\begin{proof}[of Lemma~\ref{lem:kgrand}]
By hypothesis, $k_{n+1}\geq M_{n+1}$, which implies $k_{n+1}\geq M_n$.
Let $(\delta_n)_{n\geq 1}$ be a sequence of integers such that $\delta_n = o(M_n)$ and $\frac{M_n\sqrt{\ln M_n}}{n} = o(\delta_n)$. Let $(\eta_n)_{n\geq 1}$ be another sequence of integers such that $\eta_n = o(M_n)$, $\lim_{n\to+\infty}\frac{\eta_n^2}{M_n} = +\infty$ and $\sqrt{M_n\ln (k_{n+1}-M_n)}=o(\eta_n)$.
We thus can apply Lemma~\ref{lem:u_n} {\it (ii)} to $u_n = k_{n+1}$ and conclude that, asymptotically when $n$ tends to $+\infty$,
\[B_{n,k_{n+1}} = \Theta\left(\sum_{p=M_n-\delta_n}^{\min\{M_n+\eta_n,k_{n+1}\}} a_p^{(n)}\right).\]
Moreover, since the sequence $(\delta_{n})_{n\geq 1}$ verifies $\delta_{n} = o(M_{n})$ and $\frac{M_{n}\sqrt{\ln M_{n}}}{n} = o(\delta_{n})$, and since the sequence $(\eta_n)_{n\geq 1}$ verifies $\eta_n = o(M_n)$, $\lim_{n\to+\infty}\frac{\eta_n^2}{M_n} = +\infty$ and $\sqrt{M_n\ln (k_n-M_n)}=o(\eta_n)$, we have,
\[B_{n,k_n} = \Theta\left(\sum_{p=M_n-\delta_{n}}^{\min\{M_n+\eta_n,k_{n}\}} a_p^{(n)}\right),\]
which implies
\[B_{n+1,k_{n+1}} = \Theta\left(\sum_{p=M_{n+1}-\delta_{n+1}}^{\min\{M_{n+1}+\eta_{n+1},k_{n+1}\}} a_p^{(n+1)}\right).\]
Let us note that
\[(M_{n+1}-\delta_n) \sum_{p=M_{n+1}-\delta_{n+1}}^{\min\{M_{n+1}+\eta_{n+1},k_{n+1}\}} a_p^{(n)}
\leq B_{n+1,k_{n+1}}
\leq (M_{n+1}+\eta_{n+1}) \sum_{p=M_{n+1}-\delta_{n+1}}^{\min\{M_{n+1}+\eta_{n+1},k_{n+1}\}} a_p^{(n)}.\]
Since $k_{n+1}\geq M_{n+1} \geq M_n$, via similar arguments to those developed for the proof of Lemma~\ref{lem:u_n} {\it (ii)}, we get
\[\sum_{p=M_{n+1}-\delta_{n+1}}^{\min\{M_{n+1}+\eta_{n+1},k_{n+1}\}} a_p^{(n)}
\sim \sum_{p=M_{n+1}-\delta_{n+1}}^{\min\{M_{n}+\eta_{n},k_{n+1}\}} a_p^{(n)}.\]
We thus have to compare
\[S_n = \sum_{p=M_{n+1}-\delta_{n+1}}^{\min\{M_{n}+\eta_{n},k_{n+1}\}} a_p^{(n)}\]
and
\[T_n = \sum_{p=M_n-\delta_n}^{\min\{M_n+\eta_n,k_{n+1}\}} a_p^{(n)},\]
and to prove that those two sums are equivalent when $n$ tends to infinity. Decompose $S_n$ as follows:
\[S_n = T_n 
+ \sum_{p= \min\{M_n+\eta_n,k_{n+1}\}}^{\min\{M_{n+1}+\eta_{n+1},k_{n+1}\}} a_p^{(n)} 
- \sum_{p= M_n-\delta_n}^{M_{n+1}-\delta_{n+1}}a_p^{(n)}.\]
Arguments from the proof of Lemma~\ref{lem:u_n} {\it (ii)} imply that the second summand is negligible in front of the firts. 
Let us assume that the third term is non-zero, i.e. $M_{n+1}-\delta_{n+1} > M_n-\delta_n$ (note that if this term is zero then $S_n \sim T_n$ is already proven). Via Lemma~\ref{lem:croissance}, since $\frac{M_{n+1}}{M_n} = 1+o(\frac1{M_n})$, we have
\begin{align*}
\sum_{p= M_n-\delta_n}^{M_{n+1}-\delta_{n+1}}a_p^{(n)}
&\leq (M_{n+1}-\delta_{n+1}-M_n+\delta_n) a_{M_n-\delta_n}^{(n)}\\
&= (\delta_n-\delta_{n+1} + o(1)) a_{M_n-\delta_n}^{(n)}\\
&\leq (\delta_n + o(1))a_{M_n-\delta_n}^{(n)} = o(a_{M_n}^{(n)})
\end{align*}
in view of Lemma~\ref{lem:u_n} {\it (i)}.
Therefore, since $a_{M_n}^{(n)}\leq T_n$, we have $S_n \sim T_n$ when $n$ tends to $+\infty$, which implies
\[\frac1{M_{n+1}+\eta_{n+1}} \Theta(1)\leq \rat_{n+1} := \frac{B_{n,k_{n+1}}}{B_{n+1,k_{n+1}}} \leq \frac1{M_{n+1}-\delta_{n+1}} \Theta(1),\]
and since $\eta_n = o(M_n)$ and $\delta_n = o(M_n)$, we get
\[\rat_{n+1} = \Theta\left(\frac1{M_{n+1}}\right) = \Theta\left(\frac{\ln n}{n}\right).\]
\indent\hfill\qed
\end{proof}

\section{Probability of the class of $\true$}\label{sec:tauto}

\begin{proof}[of Proposition~\ref{prop:tauto}]
The proof is divided in two steps. The first one is dedicated to the
computation of the ratio $\mu_{n}(ST)$.
The second part of the proof shows that almost all tautologies are simple tautologies.\\
Let us consider the non-ambiguous pattern language $S = \bullet | S \lor S | \boxempty\land\; \boxempty$. It is subcritical for $\mathcal I$.
Remark that a tree such that two $S$-pattern leaves are labelled by a variable and its negation,
is a simple tautology. The generating function of $S$ is $s(x,y) = \frac12(1-\sqrt{1-4(x+y^2)})$. It is sub-critical for $\mathcal{I}$. 

The generating function $\tilde{I}(z) = \frac12\nicefrac{\partial^2}{\partial x^2}(s(xz,I(z))_{|x=1}$
enumerates \ao trees with two marked distinct leaves.
Therefore, $DC_n = 2^{n-1}\tilde{I}_nB_{n-1}$ is the number of simple tautologies
where we count twice simple tautologies realized simultaneously by two pairs of leaves.
We have
\[\frac{DC_n}{T_n} = \frac{2^{n-1}\tilde{I}_{n}B_{n-1,k_n}}{2^{n+1}I_{n}B_{n,k_n}},\]
and using a consequence of~\cite[Theorem VII.8]{FS09} (cf. a detailed proof in~\cite{GK12}):
\[\lim_{n\to\infty}\frac{\tilde{I}_{n}}{I_{n}} = \lim_{z\to\frac18}\frac{\tilde{I}'(z)}{I'(z)} = 3.\]
Thus, we get the upper bound $\frac34 \rat_n$ for the ratio of simple tautologies:
it remains to deal with the double-counting in order to compute a lower bound.

In $DC_n$, simple tautologies realized by a unique pair of leaves are counted once,
those that are realized by two pairs of leaves are counted twice, and so on.
Let us denote by $ST^i_n$ the number of simple tautologies counted at least $i$ times in $DC_n$: we have $DC_n = \sum_{i\geq 1} ST_n^{(i)}$.

Our aim is to substract to $DC_n$ the tautologies that have been overcounted. Therefore, we count simple tautologies realized by three $S$-pattern leaves labelled by $\alpha/\alpha/\bar{\alpha}$ where $\alpha$ is a literal, and the tautologies realized by four $S$-pattern leaves labelled by $\alpha/\bar{\alpha}/\beta/\bar{\beta}$ where $\alpha$ and $\beta$ are two literals. Let us denote by
\[I_3(z) = \frac1{3!}\frac{\partial^3}{\partial x^3} s(xz,I(z))_{|x=1}\]
the generating function of tree-structures in which three $S$-pattern leaves have been pointed and
\[I_4(z) = \frac1{4!}\frac{\partial^4}{\partial x^4} s(xz,I(z))_{|x=1}\]
the generating function of tree-structures in which four $S$-pattern leaves have been pointed.
Then, let
\[DC_n^{(3)} = {3\cdot 2^{n-2}B_{n-2,k_n}[z^n]I_3(z) },\]
and
\[DC_n^{(4)} = {6\cdot 2^{n-2}B_{n-2,k_n}[z^n]I_4(z) }.\]
The integer $DC_n^{(3)}$ (resp. $DC_n^{(4)}$) counts (possibly with multiplicity) the trees in which three (resp. four) $S$-pattern leaves have been pointed, one of them labelled by a literal and the two others by its negation (resp. two of them labelled by two literals associated to two different variables and the two others by their negations). 
Remark that a tree having six $S$-pattern leaves labelled by $\alpha/\alpha/\bar{\alpha}/\beta/\beta/\bar{\beta}$ is counted twice by $DC_n^{(3)}$ and once by $DC_n^{(4)}$.

For all integer $i$, a simple tautology counted at least $i$ times by $DC_n$ is counted at least $(i-1)$ times by $DC_n^{(3)}+DC_n^{(4)}$. Therefore,
\[ST_n \geq DC_n - (DC_n^{(3)} + DC_n^{(4)}).\]
In view of Lemma~\ref{lem:Jakub},
\[\frac{DC_n^{(3)}}{T_n} \leq c_3\cdot \frac{B_{n-2,k_n}}{B_{n,k_n}} = c_3\cdot \frac{B_{n-2,k_n}}{B_{n,k_n}} = \mathcal{O}(\rat_n^2)\]
and
\[\frac{DC_n^{(4)}}{T_n} \leq c_4\cdot \frac{B_{n-2,k_n}}{B_{n,k_n}} = c_4\cdot \frac{B_{n-2,k_n}}{B_{n,k_n}} = \mathcal{O}(\rat_n^2),\]
where $c_3$ and $c_4$ are positive constants.
Then, asymptotically when $n$ tends to infinity, $\mu_n(ST) = \mu_n(DC) + o\left(\rat_n\right) \sim \nicefrac34\cdot \rat_n$.

Let us now turn to the second part of the proof: asymptotically, almost all tautologies are simple tautologies.
Let us consider the pattern $N = \bullet | N \lor N | \boxempty \land N$.
This pattern is unambiguous, its generating function verifies $n(x,y) = x + n(x,y)^2 + y\cdot n(x,y)$
and is thus equal to $\frac12(1-y-\sqrt{(1-y)^2-4x})$. It implies that $N$ is sub-critical for the family $\mathcal{I}$ of tree-structures.

A tautology has at least one $N[N]$-repetition, otherwise,
 we can assign all its $N$-pattern leaves to false and,
 the whole tree computes false: impossible for a tautology. \\
Consider a tautology $t$ with exactly one $N[N]$-repetition. this repetition must be a $x|\bar{x}$
 repetition and must occur among the $N$-pattern leaves, using the same kind of argument than above.\\
Then, let us assume that there is an $\land$-node denoted by $\nu$ between the $N$-pattern
 leaf $x$ and the root of the tree. This node $\nu$ has a left subtree $t_1$ and a right subtree 
$t_2$. Assume that the leaf $x$ appears in $t_1$. Then, one can assign all the
 $N$-pattern leaves of $t_2$ (which are $N[N]$-pattern leaves of $t$) to false, since there
 is no more repetition among the $N[N]$-pattern leaves of $t$. Also assign all the pattern
leaves of $t$ minus the subtree rooted at $\nu$ to false. Then, we can see that $t$ 
computes false: impossible. We have thus shown that $t$ is a simple tautology.

Finally, tautologies with exactly one $N[N]$-repetition are simple tautologies,
 a tautology must have at least one $N[N]$-repetition and, thanks to Lemma~\ref{lem:Jakub},
 tautologies with more than one $N[N]$-repetitions have a ratio of order $o\left(\rat_n\right)$, 
which is negligible in front of the ratio of simple tautologies.
\hfill\qed
\end{proof}

\section{Probability of the class of projections}\label{sec:projections}

Studying the probability of $\true$ is essential to understand the model while studying the 
projections is not necessary. However,
it permits to be more familiar with the model and often permits to conjecture
the general behaviour of $\mathbb{P}_{n}\langle f\rangle$.
This gives a sufficient reason to deeply study $\mathbb{P}_n\langle x\rangle$ ($x$ is a literal).
We will not detail all the proofs that are very similar to those of Section~\ref{sec:proba}.

To calculate the probability of the class of projections
we will follow the ideas presented for tautologies: 
we define a set of trees of simple shape that compute the projection $x$
and call such trees ``simple-$x$'' and then show that the ratio of simple-$x$ is,
asymptotically when the size of the trees $n$ tends to infinity,
equal to the probability of the projection.

\begin{definition}[cf. Figure~\ref{fig:simplex}]
A {\bf simple-$x$ of type T} is a tree with one subtree reduced to a single leaf and the other subtree being a simple tautology if the root's label is $\land$ or a simple contradiction if the root's label is $\lor$. 

A {\bf simple-$x$ of type X} is a tree with one subtree reduced to a single leaf $\ell$, the root labelled by $\land$ (resp. $\lor$) and the other subtree such that there exists a leaf labelled by the same literal as $\ell$ linked to the root by a $\lor$-only path.

We denote by $\mc{X}$ the family of simple-$x$.
\end{definition}
Obviously, simple-$x$ are computing the projection $x$.

\begin{figure}
\centering
\begin{tabular}{|cc|}
\hline
\multicolumn{2}{|c|}{Simple $x$ of type tautology.} \\
\begin{tikzpicture}[level/.style={sibling distance=40mm/#1}, level distance=10mm]
\node [circle] (z){$\vee$}
child {node [circle] (a){$x$}}
child {node [circle] (f) {$SC$}}
;
\end{tikzpicture}
&
\begin{tikzpicture}[level/.style={sibling distance=40mm/#1}, level distance=10mm]
\node [circle] (z){$\wedge$}
child {node [circle] (a){$x$}}
child {node [circle] (f) {$ST$}}
;
\end{tikzpicture}\\
\hline
\multicolumn{2}{|c|}{Simple $x$ of type $x$.}\\
\begin{tikzpicture}[level/.style={sibling distance=40mm/#1}, level distance=10mm]
\node [circle] (z){$\vee$}
child {node [circle] (a){$x$}}
child {node [circle] (f) {$\wedge$}
	child{node [circle] (g) {$\bigtriangleup$} }
	child{node [circle] (h) {$\wedge$}
		child{node [circle] (i) {$x$} }
		child{node [circle] (l) {$\bigtriangleup$} }
		}
	}
;
\end{tikzpicture}
&
\begin{tikzpicture}[level/.style={sibling distance=40mm/#1}, level distance=10mm]
\node [circle] (z){$\wedge$}
child {node [circle] (a){$x$}}
child {node [circle] (f) {$\vee$}
	child{node [circle] (g) {$\bigtriangleup$} }
	child{node [circle] (h) {$\vee$}
		child{node [circle] (i) {$x$} }
		child{node [circle] (l) {$\bigtriangleup$} }
		}
	}
;
\end{tikzpicture}\\
\hline
\end{tabular}
\caption{Examples of simple-$x$.}
\label{fig:simplex}
\end{figure}
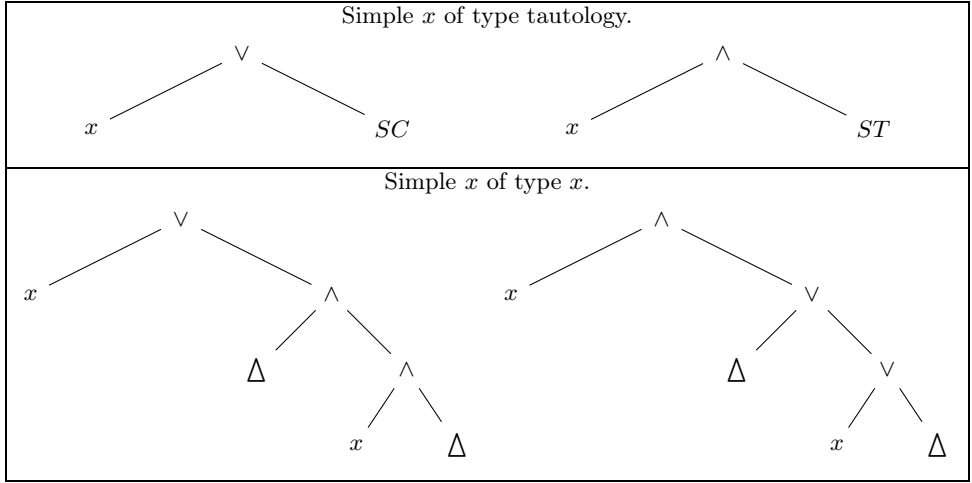

\begin{lemma}
If $X^T_n$ is the number of type T simple-$x$ of size $n$, we have, when $n$ tends to infinity:
\[\lim_{n\to+\infty}\frac{X^T_n}{T_n} \sim \frac{3}{8}\rat_n.\]
\end{lemma}
\begin{proof}
We have:
\[\frac{X^T_n}{T_{n}} \sim \frac{4\cdot 2^{n-1}B_{n-1,k_n}[z^{n-1}]\frac{\partial^2}{2\partial x^2}s(zx,I(z))_{|x=1}}{T_n}\]
because a type T simple-$x$ of size $n$ is either a tree rooted by $\land$ or a tree rooted by $\lor$ (which gives a factor 2), with either its right or its left subtree being a single leaf (which also gives a factor 2), and the other subtree being a simple tautology or a simple contradiction (depending on the root's label) of size $n-1$. 
Remark that this equation is only true asymptotically when $n$ tends to infinity, since we do double-counting which becomes negligible when $n$ tends to infinity. 
Thus, asymptotically when $n$ tends to infinity,
\[\frac{X^T_n}{T_{n,k_n}} \sim \frac{4\cdot 2^{n-1}B_{n-1,k_n}}{2^n B_{n,k_n}}\frac{[z^{n-1}]\frac{\partial^2}{2\partial x^2}s(zx,I(z))_{|x=1}}{I_n}
= \frac{2\cdot 2^{n-1}B_{n-1,k_n}}{2^n B_{n,k_n}}\frac{\tilde{I}_{n-1}}{I_n}.\]
We already have proved: $\nicefrac{\tilde{I}_n}{I_n} \sim 3$, and $\nicefrac{I_{n-1}}{I_n} = \nicefrac18$, so the result is proved.
\indent \hfill\qed
\end{proof}

\begin{lemma}
If $X^X_n$ is the number of type X simple-$x$ of size $n$, we have, asymptotically  when $n$ tends to infinity,
\[\lim_{n\to+\infty}\frac{X_n^X}{T_n} \sim \frac{\rat_n}4.\]
\end{lemma}

\begin{proof}
We have:
\[\frac{X^X_n}{T_n} \sim \frac{4\cdot 2^{n-1}B_{n-1,k_n}[z^{n-1}]\frac{\partial}{\partial x}s(zx,I(z))_{|x=1}}{2^nB_{n,k_n}I_n}\]
because a type T simple-$x$ of size $n$ is either a tree rooted by $\land$ or a tree rooted by $\lor$ (which gives a factor 2), with either its right or its left subtree being a single leaf (which also gives a factor 2), and because the other subtree is a tree where we have chosen one $S$ pattern leaf and labelled it by the same labelled as the first level leaf. Since there can be several $S$ pattern leaves that can have simultaneously the same label as the leaf subtree, we do double counting, but once again, thanks to Lemma~\ref{lem:Jakub}, this double counting becomes negligible when $n$ tends to infinity. Thus,
\[\frac{X^X_n}{T_{n}} \sim \frac{4\cdot 2^{n-1}B_{n-1,k_n}}{2^n B_{n,k_n}}\frac18.\]
Since $\nicefrac{[z^{n-1}]\frac{\partial}{\partial x}s(zx,I(z))_{|x=1}}{I_n} \sim 1$ and $\nicefrac{I_{n-1}}{I_n}\sim\nicefrac18$, we get the result.
\indent \hfill\qed
\end{proof}

\begin{lemma}
Asymptotically when $n$ tends to infinity, the ratio of simple-$x$ is equal to the probability of the projection.
\end{lemma}
Proving this lemma is very similar to proving that almost all tautology is simple (cf. proof of Proposition~\ref{prop:tauto}).

\section{Probability of a general class of Boolean functions}\label{sec:general_class}

In the following, $\langle f\rangle $ is fixed, $f$ is one of its representatives and $\Gamma_f$ is the set of essential variables of $f$.
$T$ is an \ao tree computing $f$. Moreover, we will need to consider the patterns
$R = N^{(r+1)}[N\oplus P]$ and $\bar{R} = N^{(r+1)}[(N\oplus P)^2]$.
Note that the language $N\oplus P$ is defined such that the $N\oplus P$-pattern
leaves of a tree are its $N$-pattern leaves plus its $P$ pattern leaves.
It is proved in~\cite{kozik08} that this pattern language is indeed
non-ambiguous and sub-critical for $I$ if $N$ and $P$ are.
\begin{proposition}
A tree $t$ computing $f$ with at least one leaf on the $(r+2)^{\text{th}}$ level of the $R$
pattern must have at least $R(f)+1$ $(R,\Gamma_f)$-restrictions.
\end{proposition}

\begin{proof}
Let us assume that $t$ computes $f$, has at least one leaf on the $(r+2)^{\text{th}}$ level of the $R$ pattern but have less than $R(f)$ $R$-repetitions. Let i be the smallest integer (smaller than $r+2$) such that the number of $(N^{(i)},\Gamma_f)$-restrictions is equal to the number of $(N^{(i-1)},\Gamma_f)$-restrictions.

There must be either a repetition or an essential variable in the first level: if there is none, then we can assign all the $N$ pattern leaves to $\false$ and this operation does not changes the calculated function. The calculated function is then the constant function $\false$, which is impossible;
so $i\leq r+1$.

\paragraph*{First Case:}
Let us assume that there are strictly less than $r$ $(N^{(i)},\Gamma_f)$-restrictions.
There is no repetition and no essential variable in the pattern leaves at level $i$. Therefore, we can assign them all to $\false$ and make the placeholders of the level $i-1$ compute $\false$. Let us replace those placeholders by $\false$ in the tree. Furthermore, replace by $\false$ all the non-essential remaining variables. And simplify the obtained tree to simplify all the constant leaves $\false$ and $\true$. We obtain a tree $t^{\star}$, which still computes $f$, and whose leaves are all former $N^{(i-1)}$ pattern leaves of $t$ labelled by essential variables. The tree $t^{\star}$ therefore contains strictly less than $r$ leaves, which is impossible since the complexity of $f$ is $r$.

\paragraph*{Second Case:}
Let us assume that $t$ has exactly $r$ $(N^{(i)},\Gamma_f)$-restrictions.
Since $i\leq r+1$, there is no restriction in the placeholders of the level $r+2$. Therefore, we can replace the placeholders by wildcards $\star$, which means that those wildcards can be evaluated to $\true$ or $\false$ independently from each other and without changing the function computed by $t$. We can also replace the remaining leaves labelled by non-essential and non-repeated variables by such wildcards. 

We simplify those wildcards. Such a simplification has to delete at least one non-wildcard leaf. If we deleted a non-repeated essential variable, then the tree $t^{\star}$ does not depend on this essential variable and computes $f$: this is impossible. Thus, we deleted a repetition: $t^{\star}$ has strictly less than $R(f)$ repetitions and computes $f$. It is impossible.
\hfill\qed
\end{proof}

Remark that in Lemma~\ref{lem:Jakub}, we only count repetitions and not restrictions as it was done in the original Lemma by Kozik. Because in terms of equivalence classes, essential variables are no longer relevant. Though, we will need to consider essential variables and the following lemma permits to handle them.

\begin{lemma}
\label{lem:restrictions}
Let $L$ be an unambiguous pattern, sub-critical for $\mc{I}$. Let $f$ be a fixed Boolean functions, $\Gamma_f$ the set of its  and $\mc{M}_f$ the set of minimal trees computing $f$. Let $\mc{E}$ be the family of trees obtained by expanding once a tree of $\mc{M}_f$ by trees having exactly $p$ $(L,\Gamma_f)$-restrictions. Then,
\[\mu_{n}(\mc{E}) \sim \alpha\cdot \rat_n^{R(f)+p},\]
with $\alpha>0$ a constant.
\end{lemma}

\begin{proof}
Let $E_n$ be the number of trees of size $n$ in $\mc{E}$. We will denote by $i$ the number of leaves that are involved in the $p$ $(L,\Gamma_f)$-restrictions of the expansion tree: $i$ is at least $p+1$ and at most $2p$. With negligible double-counting,
\[\mu_n(\mc{E}) = \frac{E_n}{T_n} = \sum_{i=p+1}^{2p}[z^{n-L(f)}]\frac{\partial ^i}{i!\partial x^i}\left(\ell(xz,I(z))\right)_{|x=1} \frac{2^n B_{n-p-R(f),k_n}}{2^n I_n B_{n,k_n}}.\]
Since $L$ is sub-critical for $\mc{A}$, 
\[\sum_{i=p+1}^{2p}\frac{[z^{n-L(f)}]\nicefrac{\partial ^i}{i!\partial x^i}\left(\ell(xz,I(z))\right)_{|x=1}}{I_n} \sim \alpha\cdot \frac{I_{n-L(f)}}{I_n} \sim \left(\frac18\right)^{L(f)} >0\]
asymptotically when $n$ tends to infinity. Therefore, in view of Section~\ref{sec:technical},
\[\mu_{n}(\mc{E}) \sim \alpha\cdot \rat_n^{R(f) + p}.\]
\indent \hfill\qed
\end{proof}

Consider the family of trees obtained by replacing a subtree $s$ by $s\land t_e$ where $t_e$ is a simple tautology into a minimal tree of $f$. Let us denote by $E_n$ the number of such trees of size $n$. Since a simple tautology has at least one $S$-repetition, thanks to~\ref{lem:restrictions},
\[\frac{E_n}{T_n} \sim  \alpha\cdot \rat_n^{R(f)+1}.\]

Thanks to Lemma~\ref{lem:Jakub}, we know that terms computing $f$ with more than $R(f)+2$ repetitions are negligible in front of the above family. Therefore, since trees with no leaf on the $(r+2)^{\text{th}}$ level are negligible, we proved Theorem~\ref{thm:theta}. 

In fact, we can show a more precise result:
\begin{theorem}
Let $f$ be a fixed Boolean function, then, asymptotically when $n$ tends to infinity,
\[\mathbb{P}_{n}\langle f\rangle \sim \lambda_{\langle f\rangle}\rat_n^{R(f)+1},\]
where $\lambda_{\langle f\rangle}$ is a positive constant.
\end{theorem}

The key point of the proof of this theorem is that a typical tree computing a function from $\langle f\rangle$ is a minimal tree of this function which has been expanded once. In the following, we will only consider two different expansions:
\begin{definition}[cf. Figure~\ref{fig:exp}]
Recall that an {\bf expansion} of a tree $t$ is a tree obtained by replacing a subtree $s$ of $t$ by $s\diamond t_e$ (or $t_e\diamond s$) where $\diamond\in\{\land,\lor\}$.

An expansion is a {\bf T-expansion} if the expansion tree $t_e$ is a simple tautology and the connective $\diamond$ is $\land$ (or a simple contradiction and the connective $\diamond$ is $\lor$).

An expansion is a {\bf X-expansion} if the expansion tree $t_e$ has a leaf linked to the root by a $\land$-path (resp. a $\lor$-path) and the $\diamond$ connective is a $\lor$ (resp. $\land$).
\end{definition}

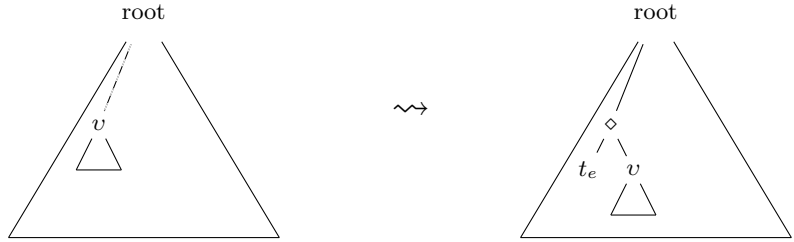
\begin{figure}[t]
\begin{center}
\begin{tabular}{m{5cm}m{0.5cm}m{5cm}}
\begin{tikzpicture}[level 1/.style={sibling distance=2cm}, level 2/.style={sibling distance=1cm,level distance=1cm}, scale=0.6]
\node [circle] (z){root}
	child[level distance=5cm]
	child[level distance=2.5cm] {node (x){$\upsilon$}
		child[level distance=1cm]
		child[level distance=1cm]
		}
	child[level distance=2.5cm] {node (m){} edge from parent[draw=none]}
	child[level distance=5cm]
;
\draw (z-1)--(z-4);
\draw[dashed,white] (z)--(x);
\draw (x-1)--(x-2);
\end{tikzpicture}
&
{\Large $\leadsto$}
&
\hspace*{1cm}
\begin{tikzpicture}[level 1/.style={sibling distance=2cm}, level 2/.style={sibling distance=1cm,level distance=1cm}, scale=0.6]
\node [circle] (z){root}
	child[level distance=5cm]
	child[level distance=2.5cm] {node (diamant){$\diamond$}
		child {node (bla){$t_e$}}
		child {node (x){$\upsilon$}
			child[level distance=1cm]
			child[level distance=1cm]
			}
		}
	child[level distance=2.5cm] {node (m){} edge from parent[draw=none]}
	child[level distance=5cm]
;
\draw (z-1)--(z-4);
\draw[dashed,white] (z)--(x);
\draw (x-1)--(x-2);
\end{tikzpicture}
\end{tabular}
\end{center}
\caption{An expansion at node $\upsilon$. Note that the expansion tree $t_e$ could have been on the right size of the $\diamond$-connective instead of its left side.}
\label{fig:exp}
\end{figure}

\begin{lemma}
The ratio of minimal trees of $f$ expanded once verifies, asymptotically when $n$ tends to infinity
\[\mu_{n}(E[\mc{M}_f]) = \alpha\cdot \rat_n^{R(f)+1} + o\left(\rat_n^{R(f)+1}\right).\]
\end{lemma}
This lemma is a direct consequence of Lemma~\ref{lem:restrictions}.


\begin{lemma}
Let $f$ be a fixed Boolean function, $\Gamma_f$ the set of its essential variables and $\mc{M}_f$ the set of minimal trees of $f$.
\[\mathbb{P}_n\langle f\rangle \sim \mu_n(E[\mc{M}_f])\text{ when }n\to+\infty.\]
\end{lemma}

\begin{proof}
Let $t$ be a term computing $f$. Such a term must have at least $R(f)+1$ $\bar{R}$-repetitions. Moreover, thanks to Lemma~\ref{lem:Jakub}, trees with at least $R(f)+2$ $\bar{R}$-repetitions are negligible. We will show that a tree with exactly $R(f)+1$ $\bar{R}$-repetitions is in fact a minimal tree expanded once.

The term $t$ must also have $R(f)+1$ $R$-repetitions and therefore, there is no additional repetition when we consider the $(r+3)^{\text{st}}$ level of the $\bar{R}$-pattern.

Let $i$ be the first level such that the number of $(N^{(i)},\Gamma_f)$-restrictions is equal to the number of $N^{(i-1)}$-restrictions. Since there must be a restriction on the first level, $i\leq r+1$.

\paragraph*{First Case:}Assume that an essential variable $\alpha$ appears on the pattern leaves of the $(r+3)^{\text{th}}$ level.  
Therefore, $t$ has at most $L(f)$ $(N^{(i)},\Gamma_f)$-restrictions. Let us replace the placeholders of the $(i-1)^{\text{th}}$ level by $\false$ and assign all the remaining non-essential variables to $\false$. Simplify the tree to obtain a new and/or tree denoted by $t^{\star}$. The leaves of this tree are former $N^{(i-1)}$-pattern leaves of $t$, labelled by essential variables and $t^{\star}$ still computes $f$. But the variable $\alpha$ is essential for $f$: thus it must still appear in the leaves of $t^{\star}$, and by deleting its occurence in the leaves of the $(r+3)^{\text{th}}$ level, we deleted one repetition. Therefore, $t^{\star}$ has at most $L(f)-1$ leaves which is impossible!

\paragraph*{Second Case:}There is no essential variable among the the pattern leaves of the $(r+3)^{\text{th}}$ level. Since there is also no repetition at this level, we can replace the placeholders of the level $(r+3)$ to wildcards. We also replace the remaining non essential and non-repeated variables by wildcards. We then simplify the wildcards and obtained a simplified tree $t^{\star}$, computing $f$, with no wildcards and which leaves are former leaves of the trees $t$, essential or repeated. During the simplification process, we have deleted at least one of these leaves and therefore $t^{\star}$ has at most $L(f)$ leaves: it is a minimal tree of $f$.\\
Let us consider the following fact:
The lowest common ancestor of all the wildcards in $t$ has
been suppressed during the simplification process.\\
Assume that this fact is false: then two wildcards have been simplified independently during the simplification process, and thus, at least two essential or repeated variables have been deleted. The tree $t^{\star}$ has thus at most $L(f)-1$ leaves and computes $f$, which is impossible since $L(f)$ is the complexity of $f$.

Let us denote by $t_e$ the subtree rooted at $\upsilon$ the lowest common ancestor of the wildcards. Thus a typical tree computing $f$ is a minimal tree of $f$ in which we have plugged a specific expansion tree $t_e$.
\hfill\qed
\end{proof}

\begin{lemma}
Let $t$ be a typical tree computing $f$. The expansion tree $t_e$ is either a simple tautology (or simple contradiction), or an $x$-expansion - i.e. a tree with one $\land$-leaf (resp. $\lor$-leaf) labelled by an essential variable of $f$.
\end{lemma}

\begin{proof}
As shown in the former lemma, a typical tree computing $f$ is a minimal tree of $f$ on which has been plugged an expansion tree $t_e$.
\paragraph*{First Case:}Let us assume that $t_e$ has no $(N\oplus P)$-repetition and no essential variable among its $(N\oplus P)$-pattern leaves. Then, we can replace $t_e$ by a wildcard and simplify this wildcard. This simplification suppresses at least one other leaf of the tree: the obtained tree is then smaller than the original minimal tree, and still computes $f$. It is impossible.
\paragraph*{Second Case:}Let us assume that $t_e$ has at least two $((N\oplus P)^2,\Gamma_f)$-restrictions. Thanks to Lemma~\ref{lem:restrictions}, this family of expanded trees is negligible.
\paragraph*{Third Case:}Let us assume that $t_e$ has exactly one $((N\oplus P)^2,\Gamma_f)$-restrictions. Then it must be a $(N\oplus P,\Gamma_f)$-restriction (cf. First Case).
\begin{itemize}
\item if it is a repetition, than one can show that it must be a simple tautology or a simple contradiction.
\item if it is an essential variable, one can show that it must be an $X$-expansion.
\end{itemize}
\indent \hfill\qed
\end{proof}

\end{document}